\documentclass[10pt]{amsart}

\usepackage{lmodern}
\usepackage[T1]{fontenc}

\usepackage{amsmath, amsthm, amscd, amssymb, xcolor}
\definecolor{dblue}{rgb}{0,0,.6}
\usepackage[colorlinks=true, linkcolor=dblue, citecolor=dblue, filecolor = dblue, menucolor = dblue, urlcolor = dblue]{hyperref}
\usepackage{tikz-cd} 
\usepackage[all,cmtip]{xy}

\newcommand{\bbQ}{\mathbb{Q}}
\newcommand{\bbR}{\mathbb{R}}
\newcommand{\bbC}{\mathbb{C}}
\newcommand{\bbH}{\mathbb{H}}

\newcommand{\CC}{\mathcal{C}}

\newcommand{\GG}{\mathcal{G}}

\newcommand{\JJ}{\mathcal{J}}

\newcommand{\frgt}{\mathfrak{g}_{\mathrm{tot}}}

\newcommand{\frsl}{\mathfrak{sl}}

\newcommand{\frso}{\mathfrak{so}}

\renewcommand{\ge}{\geqslant}
\renewcommand{\le}{\leqslant}

\newcommand{\td}{\mathrm{td}}
\newcommand{\rtd}[1]{\sqrt{\td(#1)}}

\newcommand{\st}{\enskip |\enskip}
\newcommand{\sdot}{{\raisebox{0.16ex}{$\scriptscriptstyle\bullet$}}}

\newcommand{\emrp}{\mathrm{End}}

\newcommand{\ii}{i}

\newcommand{\Cl}{\mathcal{C}l}

\newcommand{\Spin}{\mathrm{Spin}}
\newcommand{\lrarr}{\longrightarrow}
\newcommand{\hrarr}{\hookrightarrow}

\newtheorem{defn}{Definition}[section]
\newtheorem{prop}[defn]{Proposition}
\newtheorem{thm}[defn]{Theorem}

\newtheorem{cor}[defn]{Corollary}

{\theoremstyle{remark}
  \newtheorem{rem}[defn]{Remark}
  
}

\makeatletter
\@addtoreset{equation}{section}
\makeatother

\title{On the Hodge structures of compact hyperk\"ahler manifolds}

\author{Andrey Soldatenkov}
\address{Institut f\"ur Mathematik, Humboldt-Universit\"at zu Berlin, Unter den Linden 6, 10099 Berlin}
\email{soldatea@hu-berlin.de}
\date{\today}
\subjclass[2010]{primary 14J32; secondary 14C30} 

\thanks{}

\begin{document}

\begin{abstract}
The purpose of this note is to give an account of a well-known
folklore result: the Hodge structure on the second cohomology of a compact
hyperk\"ahler manifold uniquely determines Hodge structures
on all higher cohomology groups. We discuss the precise statement and
its proof, which are somewhat difficult to locate in the literature.
\end{abstract}

\maketitle

%\tableofcontents

\section{Introduction}

Compact hyperk\"ahler manifolds have been extensively studied
in recent decades. One of the central results of their theory
is the global Torelli theorem \cite{V4}. It addresses the problem
of reconstructing a hyperk\"ahler manifold from the Hodge structure
on its second cohomology group. It is known that in general one
can not reconstruct the manifold uniquely, and the global Torelli theorem
explains the reasons for this. It gives a description of the moduli
space of hyperk\"ahler manifolds as a certain non-Hausdorff covering space
of the period domain for the Hodge structures on the second cohomology group,
see e.g. the discussion in \cite{H3}.

Despite of the fact that it is impossible to reconstruct a hyperk\"ahler
manifold from the Hodge structure on $H^2$, one can still ask if it is possible
to recover the rational Hodge structures on higher cohomology groups from the Hodge
structure on $H^2$. It turns out that in a certain sense this is possible,
and such statements have appeared in the literature (e.g. in the preprint
version of \cite{LL} or \cite[Corollary 24.5]{GHJ}). In this note we prove a
more precise version of this result, Theorem \ref{thm_main}. A more standard version
is stated as Corollary \ref{cor_main}.
Let us remark that the proof of Theorem \ref{thm_main} does not use
the global Torelli theorem.
% and does not use any information about the monodromy
% group of a hyperk\"ahler manifold.

In section \ref{sec_coho} we recall all necessary definitions and results about the structure
of the cohomology algebra of hyperk\"ahler manifolds and sketch some
of the proofs. In section \ref{sec_main} we discuss sufficient conditions
for a complex structure to be of hyperk\"ahler type, Proposition \ref{prop_hk_type}.
In the end we prove the main result, Theorem \ref{thm_main}.

\subsection*{Acknowledgements} The write-up of this note was encouraged by Daniel Huybrechts, who has
repea\-tedly inquired the author about the proofs of the discussed statements.
I am very grateful for his interest.

\section{Cohomology of hyperk\"ahler manifolds}\label{sec_coho}

\subsection{Topological invariants}

Let $X$ be a compact $C^\infty$-manifold, $\dim_\bbR(X) = 2n$.
The singular cohomology of $X$ with rational coefficients $H^\sdot(X,\bbQ)$
is a finite-dimensional graded $\bbQ$-algebra.
We will denote by $p_i(X)\in H^{4i}(X,\bbQ)$ the rational
Pontryagin classes of $X$.

\begin{defn}\label{defn_groups}
Consider the following groups:
\begin{enumerate}
\item $\GG(X)$ -- automorphisms of the
graded algebra $H^\sdot(X,\bbQ)$ that stabilize all $p_i(X)$;
\item $\GG^+(X)$ -- automorphisms of the graded subalgebra $H^{2\sdot}(X,\bbQ)$
that stabilize all $p_i(X)$.
\end{enumerate}
\end{defn}

Define the operator $\theta\in\emrp(H^\sdot(X,\bbQ))$ as follows:
$$
\theta|_{H^{k}(X,\bbQ)} = (k-n)\mathrm{Id}.
$$
For an element $h\in H^2(X,\bbQ)$, let $L_h\in\emrp(H^\sdot(X,\bbQ))$
denote the operator of cup product with $h$. We will say that $h$ has Lefschetz property, if
$$
L_h^k\colon H^{n-k}(X,\bbQ)\stackrel{\sim}{\lrarr} H^{n+k}(X,\bbQ)
$$
is an isomorphism for all $k=0,\ldots,n$.
If $h$ has Lefschetz property, then there exists a unique $\Lambda_h\in \emrp(H^\sdot(X,\bbQ))$,
such that $(\Lambda_h,\theta,L_h)$ is an $\frsl_2$-triple.

\begin{defn}\label{defn_gtot}
Let us denote by $\frgt(X)$ the minimal
Lie subalgebra of $\emrp(H^\sdot(X,\bbQ))$ containing $\theta$, $\Lambda_h$ and $L_h$ for all
$h\in H^{2}(X,\bbQ)$ with the Lefschetz property.
\end{defn}

\begin{rem}
The groups $\GG(X)$, $\GG^+(X)$ and the Lie algebra $\frgt(X)$ depend only on the homeomorphism type of $X$.
For $\frgt(X)$ this is clear from the definition, and for $\GG(X)$, $\GG^+(X)$ it follows
from a theorem of Novikov \cite{No}.
\end{rem}

We will use the following notations. The $\bbR$-Lie algebra $\frgt(X)\otimes_{\bbQ}\bbR$ will be
denoted by $\frgt(X)_\bbR$. For any $\psi\in \GG(X)$ we have $\psi \theta \psi^{-1} = \theta$, $\psi L_h \psi^{-1} = L_{\psi(h)}$ and hence
$\psi \Lambda_h \psi^{-1} = \Lambda_{\psi(h)}$ for all $h\in H^2(X,\bbQ)$. This shows that
the adjoint action of $\psi$ preserves $\frgt(X)$. We will denote by $\mathrm{ad}_\psi$
the corresponding endomorphism of $\frgt(X)$.

\subsection{Hyperk\"ahler manifolds}

Given a complex structure $I\in\emrp(TX)$,
we will denote by $X_I$ the corresponding complex manifold, and by $\Omega^k_{X_I}$
the sheaves of holomorphic differential forms on $X_I$.
The canonical bundle will be denoted by $K_{X_I}$.

\begin{defn}\label{defn_hk}
Assume that the manifold $X$ is compact and $\pi_1(X) = 1$.
We will say that a complex structure $I$ is of hyperk\"ahler type, if
\begin{enumerate}
\item $X_I$ admits a K\"ahler metric;
\item $H^0(X_I,\Omega_{X_I}^2)$ is spanned by a symplectic form.
\end{enumerate}
In this case $X_I$ is called a hyperk\"ahler manifold. We will say
that $X$ is of hyperk\"ahler type, if it admits a complex structure of hyperk\"ahler type.
\end{defn}

Assume that $I$ is of hyperk\"ahler type, and let $\sigma\in H^0(X_I,\Omega_{X_I}^2)$ be a symplectic form.
The dimension of any symplectic manifold is even, and we let $2n = \dim_{\bbC}(X_I)$.
The form $\sigma$ defines an isomorphism $T_{X_I} \simeq \Omega^1_{X_I}$, which
shows that all odd Chern classes of $T_{X_I}$ vanish. The total Todd class
of a complex vector bundle with vanishing odd Chern classes can be expressed as a universal
polynomial in the Pontryagin classes of the underlying real bundle.
Evaluating this polynomial on the Pontryagin classes of $X$ gives an element $\td(X)\in H^{4\sdot}(X,\bbQ)$
that does not depend on the choice of the complex structure $I$.
This element is the total Todd class of $X_I$ for any $I$ of hyperk\"ahler type.
Since $\td^0(X) = 1$, there is a unique square root of $\td(X)$ with
degree zero term equal to $1$, and we denote it by $\rtd{X}\in H^{4\sdot}(X,\bbQ)$.

It was shown in \cite{HS}, that $\int_{X_I}\rtd{X} > 0$. The integral here means
evaluation of the degree $4n$ component of $\rtd{X}$ on the fundamental class of $X$.
The latter is determined by the orientation of $X$ induced by $I$. In particular,
this shows that all complex structures of hyperk\"ahler type induce the same
orientation on $X$, and that all diffeomorphisms of $X$ are orientation-preserving,
since they have to fix all polynomial expressions in Pontryagin classes.
From now on we will implicitly assume that we have fixed the orientation of $X$.

\begin{defn}\label{defn_BBF}
The Beauville-Bogomolov-Fujiki (BBF) form of $X$ is the quadratic form $q\in S^2H^2(X,\bbQ)^*$ given by
$$
q(a) = \int_X a^2\rtd{X}
$$
for all $a\in H^2(X,\bbQ)$.
\end{defn}

\begin{rem}
Usually the BBF form is defined via the Fujiki relations (\ref{eqn_fujiki2}) that we recall below.
We prefer the above definition to avoid the ambiguity in the choice of the constant in (\ref{eqn_fujiki2}).
The fact that the definition above is equivalent up to a scalar factor to the usual one
is due to \cite{Ni}, see also the discussion in \cite[section 4]{H2}.
%One can rescale $q$ by a positive rational factor to make it integral and primitive on the
%intergal cohomology of $X$. 
\end{rem}

Let us list a few properties of the BBF form.
\begin{enumerate}

\item The form $q$ is non-degenerate of signature $(3,b_2(X)-3)$. If $\omega\in H^2(X,\bbR)$ is a K\"ahler class
for a complex structure of hyperk\"ahler type, then $q(\omega)>0$, see \cite[Th\'eor\`eme 5 and p. 773]{Be} and \cite[Theorem 4.2]{H2}.

\item For every $k = 1,\ldots,n$ there exists a non-zero constant $C_{X,k}\in \bbQ$, such that for all $a\in H^2(X,\bbQ)$
\begin{equation}\label{eqn_fujiki1}
\int_X a^{2k}\rtd{X} = C_{X,k} q(a)^k.
\end{equation}
This follows from \cite[Theorem 4.2]{H2} and the inequality $\int_X\rtd{X} > 0$
from \cite{HS}. In particular, for $k = n$ we get the Fujiki relation:
\begin{equation}\label{eqn_fujiki2}
\int_X a^{2n} = C_{X,n} q(a)^n.
\end{equation}

\item For all $a,b\in H^2(X,\bbC)$ we have:
\begin{equation}\label{eqn_BBF1}
\int_X a^{2n-1}b = C_{X,n}q(a)^{n-1}q(a,b).
\end{equation}
This relation follows from (\ref{eqn_fujiki2}) by substituting $a+tb$ in place of $a$
and comparing the coefficients of the obtained polynomials in $t$.

\item For all $a,b\in H^2(X,\bbC)$ such that $q(a,b)=0$, we have:
\begin{equation}\label{eqn_BBF2}
(2n-1)\int_X a^{2n-2}b^2 = C_{X,n}q(a)^{n-1}q(b).
\end{equation}
This follows from (\ref{eqn_fujiki2}) or from \cite[Th\'eor\`eme 5c]{Be}.

\item Let $\mathcal{I} \subset S^\sdot H^2(X,\bbC)$ denote the ideal generated by $a^{n+1}$ for all $a\in H^2(X,\bbC)$ with $q(a) = 0$.
Then according to \cite[Theorem 2.4 and Lemma 2.2]{B1} the multiplication in cohomology induces an embedding
\begin{equation}\label{eqn_sym}
S^\sdot H^2(X,\bbC)/\mathcal{I} \hrarr H^\sdot (X,\bbC).
\end{equation}

\end{enumerate}

\subsection{The Lie algebra action}\label{sec_lie_action}
We assume that $X_I$ is a hyperk\"ahler manifold. 
It follows from Calabi's conjecture proven by Yau,
that in this case $X$ admits two other complex structures $J$, $K$ and a Riemannian metric $g$,
such that $K = IJ = -JI$ and $g$ is K\"ahler with respect to $I$, $J$ and $K$, see e.g. \cite{Be} or \cite{GHJ}.
We will use the following notation: $\omega_I$, $\omega_J$ and $\omega_K$ will denote
the K\"ahler forms, $L_I$, $L_J$ and $L_K$ the corresponding Lefschetz operators,
$\Lambda_I$, $\Lambda_J$ and $\Lambda_K$ the dual Lefschetz operators. The complex
structures can be extended as derivations to act on the differential $k$-forms on $X$
for all $k$.
The corresponding operators will be denoted by $W_I$, $W_J$ and $W_K$. For instance, $W_I$
acts on the differential forms of $I$-type $(p,q)$ as multiplication by $\ii (p-q)$.
In the proposition below, $\Lambda^\sdot T^*\!X$ denotes the graded vector bundle of
real differential forms on $X$, and $\emrp(\Lambda^\sdot T^*\!X)$ denotes the algebra
of its endomorphisms.

\begin{prop}\label{prop_alg1}
The Lie subalgebra of $\emrp(\Lambda^\sdot T^*\!X)$ generated by the operators
$L_I$, $L_J$, $L_K$, $\Lambda_I$, $\Lambda_J$ and $\Lambda_K$ is isomorphic
to $\frso(4,1)$. We have the following commutator identities:
$$
[\Lambda_I, L_J] = W_K; \quad [\Lambda_J, L_K] = W_I; \quad [\Lambda_K, L_I] = W_J;
$$
$$
[\Lambda_I, \Lambda_J] = [\Lambda_J, \Lambda_K] = [\Lambda_K, \Lambda_I] = 0.
$$
\end{prop}
\begin{proof}
The proof of this statement can be found in \cite[Theorem 8.1]{V3}, see also references therein.
We sketch an alternative proof, based on the theory of $k$-symplectic structures from \cite{KSV}.

It clearly suffices to prove the commutator identities pointwise, so we are reduced to the following
linear-algebraic problem. Consider $M = \bbH^{\oplus n}$ as a left $\bbH$-module with the
standard metric $g(x,y) = \sum_{k=1}^n x_k\bar{y}_k$, where the bar denotes quaternionic conjugation.
We have the operators of multiplication by imaginary quaternions $I,J,K\in\emrp(M)$
and the corresponding two-forms $\omega_I,\omega_J,\omega_K\in \Lambda^2M^*$. We need to prove
the commutator identities for the Lefschetz operators and their duals in $\emrp(\Lambda^\sdot M^*)$

Let $U\subset \bbH$ be the three-dimensional subspace of imaginary quaternions with the quadratic form
$\rho(a) = \mathrm{Re}(a^2)$. The Clifford algebra $\CC = \Cl(U,\rho)$ is by construction endowed
with a natural morphism $\CC\to \bbH$, making $M$ a left $\CC$-module. The metric $g$ is
a $\CC$-invariant symmetric bilinear form in the sense of \cite[Definition 3.1]{KSV}.
We define a map $\eta\colon U\to \Lambda^2M^*$ by sending $a\in U$ to the form $\omega_a$,
such that $\omega_a(x,y) = g(ax,y)$. The complexification of the image of $\eta$ is a 3-symplectic structure
on $M_\bbC$ in the sense of \cite[Definition 1.1]{KSV}. It is clear that the image of $\eta$ is the
linear span of $\omega_I$, $\omega_J$ and $\omega_K$. The statement now follows
from \cite[Theorem 3.10 and Lemma 3.12]{KSV}.
\end{proof}

It is known that the Lefschetz operators and their duals commute with the Laplacian
of the Riemannian metric $g$. Hence all the operators from the above proposition act on the cohomology of $X$,
and we obtain an embedding of Lie algebras $\frso(4,1)\hrarr\frgt(X)_\bbR$. This embedding
depends on the choice of the complex structure $I$ and the hyperk\"ahler metric $g$.
Using local deformation theory of complex structures on $X$, we can obtain enough
$\frso(4,1)$-subalgebras in $\frgt(X)_\bbR$ to conclude that all dual Lefschetz operators
on $X$ pairwise commute. This observation leads to the description of $\frgt(X)$ that we give below.

\begin{defn}
Let us denote by $V$ the $\bbQ$-vector space $H^2(X,\bbQ)$. Define the graded $\bbQ$-vector
space $\tilde{V} = \langle e_0\rangle \oplus V\oplus \langle e_4\rangle$ with $e_k$ of degree $k$
and $V$ in degree 2. Define the quadratic form $\tilde{q}\in S^2\tilde{V}^*$, such that $\tilde{q}|_V = q$,
and $\langle e_0, e_4\rangle$ is a hyperbolic plane orthogonal to $V$ with $q(e_0)=q(e_4)=0$ and $q(e_0,e_4) = 1$.
\end{defn}
The graded Lie algebra $\frso(\tilde{V},\tilde{q})$ has components of degrees $-2$, $0$ and $2$.
The semisimple part of $\frso(\tilde{V},\tilde{q})^0$ is isomorphic to $\frso(V,q)$, and we
have the following isomorphisms of $\frso(V,q)$-modules: $\frso(\tilde{V},\tilde{q})^{-2} \simeq \frso(\tilde{V},\tilde{q})^{2} \simeq V$,
see e.g. \cite[section 3.4]{KSV}.

\begin{prop}\label{prop_gtot}
There exists an isomorphism of graded Lie algebras $\frgt(X) \simeq \frso(\tilde{V},\tilde{q})$.
The subalgebra $\frso(V,q) \subset \frgt(X)$ acts on $H^\sdot(X,\bbQ)$ by derivations.
\end{prop}
\begin{proof}
It is proven in \cite[Proposition 4.5]{LL} that $\frgt(X)_\bbR \simeq \frso(\tilde{V}_\bbR,\tilde{q})$.
To deduce the corresponding statement over $\bbQ$, note that under the embedding
$\frgt(X)_\bbR\subset \emrp(H^\sdot(X,\bbR))$ the components $\frso(\tilde{V}_\bbR,\tilde{q})^2$
and $\frso(\tilde{V}_\bbR,\tilde{q})^{-2}$ are mapped to the subspaces of Lefschetz operators,
respectively dual Lefschetz operators. These embeddings are defined over $\bbQ$, since
for $x\in H^2(X,\bbQ)$ with $q(x)\neq 0$ both $L_x$ and $\Lambda_x$ are defined over $\bbQ$.
Since $\frso(\tilde{V},\tilde{q})$ is generated by the components of degree $\pm 2$,
this proves the first statement of the proposition. The second statement follows directly
from \cite[Proposition 4.5]{LL}.
\end{proof}

It follows from Proposition \ref{prop_gtot} that there exists a representation
of $\Spin(V,q)$ in the group of algebra automorphisms of $H^\sdot(X,\bbQ)$.
Recall Definition \ref{defn_groups} of the groups $\GG(X)$
and $\GG^+(X)$, and observe that there exists a natural homomorphism
$\GG(X) \to \GG^+(X)$. 

\begin{prop}\label{prop_groups}
The action of the group $\Spin(V,q)$ on the algebra $H^\sdot(X,\bbQ)$ obtained from
Proposition \ref{prop_gtot} is induced by a homomorphism
$\alpha\colon \Spin(V,q)\to \GG(X)$. The action of $\Spin(V,q)$ on
$H^{2\sdot}(X,\bbQ)$ factors through $\mathrm{SO}(V,q)$:
\begin{equation}\label{eqn_groups}
\begin{tikzcd}[]
\Spin(V,q) \rar{\alpha}\arrow{d} & \GG(X) \arrow{d} \\
\mathrm{SO}(V,q) \rar{\alpha^+} & \GG^+(X)
\end{tikzcd}\nonumber
\end{equation}
\end{prop}
\begin{proof}

Since the Pontryagin classes of $X$ are of Hodge type $(p,p)$ for all
complex structures admitting a K\"ahler metric, one deduces that
$\Spin(V,q)$ fixes all the Pontryagin classes, see \cite[Proposition 4.8]{LL}.
This gives a homomorphism $\alpha$. It was shown in \cite[Corollary 8.2]{V3}
that the composition of $\alpha$ and the homomorphism $\GG(X) \to \GG^+(X)$
factors through $\mathrm{SO}(V,q)$.
\end{proof}

\begin{prop}\label{prop_grhomo}
For an element $\psi\in \GG(X)$ let $\varphi = \psi_2$ denote its degree two component acting
on $V = H^2(X,\bbQ)$. Then $\varphi\in \mathrm{O}(V,q)$. For any $x\in \frso(V,q)\subset \frgt(X)^0$,
we have $\mathrm{ad}_\psi(x) = \mathrm{ad}_{\varphi}(x)$.
\end{prop}
\begin{proof}
The first statement follows from Definition \ref{defn_BBF}, because $\GG(X)$ fixes
all the Pontryagin classes, and $H^{4n}(X,\bbQ)$ is spanned by a polynomial
in the Pontryagin classes (see the paragraph before Definition \ref{defn_BBF}).

For the second statement, consider the composition of the inclusion $\frso(V,q)\subset \frgt(X)$
and the homomorphism $\frgt(X)\to \emrp(V)$ obtained by restricting the action of $\frgt(X)$ 
to $H^2(X,\bbQ)$.
This composition equals the canonical embedding $\frso(V,q)\subset \emrp(V)$ (see e.g. \cite[Claim 1 on p. 392]{LL}), so the adjoint
action of $\psi$ on $\frso(V,q)$ is determined by the action of its degree two component,
which is $\varphi$.
\end{proof}

\section{Hodge structures on the cohomology of hyperk\"ahler manifolds}\label{sec_main}

\subsection{Complex structures of hyperk\"ahler type}

If $I_1$ and $I_2$ are two complex structures on $X$, and $I_1$ is
of hyperk\"ahler type, it is not a priory clear that $I_2$ is also
of hyperk\"ahler type. Two conditions are necessary for this:
$I_2$ should admit a K\"ahler metric, and the canonical bundle of $X_{I_2}$
should be trivial. The following
lemma shows that these conditions are also sufficient under a technical
assumption on $b_2(X)$.

\begin{prop}\label{prop_hk_type}
Assume that $X$ is of hyperk\"ahler type with $b_2(X)\ge 5$.
Let $I$ be an arbitrary complex structure on $X$.
The following conditions are equivalent:
\begin{enumerate}
\item $I$ is of hyperk\"ahler type;
\item $I$ admits a K\"ahler metric and $c_1(K_{X_I}) = 0$.
\end{enumerate}
\end{prop}
\begin{proof}
Since the top exterior power of a symplectic form trivializes the canonical bundle, the implication (1)$\Rightarrow$(2)
is obvious. Let us prove the converse.

According to the decomposition theorem of Bogomolov \cite{B2},
$$
X_I \simeq Y\times Z_1\times\ldots\times Z_m,
$$ 
where $Y$ is a Calabi-Yau manifold with $h^{2,0}(Y) = 0$ and
$Z_i$ are hyperk\"ahler manifolds in the sense of Definition \ref{defn_hk}. Let $\dim_\bbC(X_I)= 2n$ and
$\dim_\bbC(Z_i)= 2n_i$. 

It follows from (\ref{eqn_sym}) that the multiplication map $S^nH^2(X,\bbC) \to H^{2n}(X,\bbC)$
is injective. Assume that $n_i<n$ for some $i$. Let $\pi\colon X_I\to Z_i$ be the
projection and $\sigma\in H^0(Z_i,\Omega^2_{Z_i})$ be the symplectic form.
Then $(\pi^*\sigma)^n = \pi^*(\sigma^n) = 0$, which is a contradiction.
We conclude that $m\le 1$, and if $m=1$, then $X_I\simeq Z_1$.

It remains to exclude the case $m=0$, i.e. $X_I\simeq Y$. Assuming
that this is the case, let $\omega\in H^2(X,\bbR)$ be a K\"ahler class for $I$
and $H^2_{\omega}(X,\bbR) = \{a\in H^2(X,\bbR) \st \int_X\omega^{2n-1}a = 0\}$.
It follows from (\ref{eqn_fujiki2}) that $q(\omega)\neq 0$.
The equation (\ref{eqn_BBF1}) shows that $H^2_{\omega}(X,\bbR)$ is the $q$-orthogonal
complement of $\omega$. Since we assume that $h^{2,0}(X_I) = 0$, the Hodge-Riemann
bilinear relations and the formula (\ref{eqn_BBF2}) with $a=\omega$ imply that $q$ is sign-definite
on $H^2_{\omega}(X,\bbR)$. Since the signature of $q$ is $(3, b_2(X)-3)$,
this contradicts our assumptions on $b_2(X)$. This completes the proof.
\end{proof}

\subsection{Hodge structures}
As before, we will denote by $q$ the BBF form on $V = H^2(X,\bbQ)$, see Definition \ref{defn_BBF}.
Let $I_1$, $I_2$ be complex structures of hyperk\"ahler type on $X$. Then $H^2(X_{I_1},\bbQ)$ and
$H^2(X_{I_2},\bbQ)$ are rational Hodge structures having the same underlying vector space $V$.

%Below we will use the following notation: if $G$ is a $\bbQ$-algebraic group,
%then $G_{\bbQ}$ is the group of its rational points.
\begin{defn}\label{defn_iso}
A rational Hodge isometry between $H^2(X_{I_1},\bbQ)$ and $H^2(X_{I_2},\bbQ)$ is an element
$\varphi\in \mathrm{O}(V,q)$, such that $\varphi(H^{p,q}(X_{I_1})) = H^{p,q}(X_{I_2})$
for all $p + q = 2$.
\end{defn}

\begin{defn}\label{defn_subgr}
Define the following subgroups of $\mathrm{O}(V,q)$:
\begin{enumerate}
\item
$\mathcal{J}^+$ is the image of the homomorphism
$\GG^+(X)\to \mathrm{O}(V,q)$ (see Proposition \ref{prop_grhomo});
\item
$\mathcal{J}\subset \mathcal{J^+}$ is the image of the composition
$\GG(X)\to \GG^+(X)\to \mathrm{O}(V,q)$.
\end{enumerate}
\end{defn}

We are interested in Hodge isometries that are contained either in $\JJ$ or in $\JJ^+$.
Let us give some sufficient conditions for an isometry of $(V,q)$ to be
contained in one of these groups. Recall that there exists a group homomorphism
$$
SN\colon \mathrm{SO}(V,q) \to \bbQ^\times/(\bbQ^\times)^2,
$$
called the spinor norm, and that
$$
\mathrm{Ker}(SN) = \mathrm{Im}(\Spin(V,q)\to \mathrm{SO}(V,q)),
$$ see e.g. \cite[Abschnitt 8]{Kn}.

\begin{prop}\label{prop_cond} We have the following inclusions:
\begin{enumerate}
\item $\mathrm{SO}(V,q)\subset \JJ^+$;
\item $\mathrm{Ker}(SN) \subset \JJ$.
\end{enumerate}
\end{prop}
\begin{proof}
Both inclusions easily follow from the definitions and Proposition \ref{prop_groups}.
\end{proof}

\begin{rem}
The inclusions in Proposition \ref{prop_cond} are in general strict.
For example, any diffeomorphism $\Phi\colon X\to X$ induces an isometry
$\varphi = \Phi^*$ of $(V,q)$, and $\varphi\in\JJ$.
But such $\varphi$ does not in general preserve the orientation on $V$,
so does not always lie in $\mathrm{SO}(V,q)$.

Since $\mathrm{SO}(V,q)$ is of index two in $\mathrm{O}(V,q)$,
it is enough to produce one element of $\JJ$ that does not preserve the
orientation on $V$ to prove that $\JJ^+ = \mathrm{O}(V,q)$.
For all known examples of compact hyperk\"ahler manifolds
one can do that, because their monodromy group (see \cite[Definition 1.1]{Ma})
contains reflections along the classes of prime exceptional divisors
(see \cite[Definition 5.1]{Ma}). For varieties of $\mathrm{K3}^{[n]}$ type, generalized
Kummer type and O'Grady's 10-dimensional example, see Theorem 9.1 and
two paragraphs after Conjecture 10.6 in \cite{Ma}. For O'Grady's 6-dimensional
example the existence of a prime exceptional divisor follows from \cite{Na}.
\end{rem}

We can now state the main result.

\begin{thm}\label{thm_main}
Let $I_1$ and $I_2$ be two complex structures of hyperk\"ahler type on a compact
simply-connected manifold $X$ with $\dim_{\bbR}(X) = 4n$. Assume that there exists a rational Hodge isometry
$$
\varphi\colon H^2(X_{I_1},\bbQ) \stackrel{\sim}{\lrarr} H^2(X_{I_2},\bbQ).
$$

\begin{enumerate}
\item If $\varphi\in \JJ$, then there exists an isomorphism of rational Hodge structures
$$
\psi\colon H^\sdot(X_{I_1},\bbQ) \stackrel{\sim}{\lrarr} H^\sdot(X_{I_2},\bbQ)
$$
that extends $\varphi$, respects the grading and the algebra structure;

\item If $\varphi\in \JJ^+$, then there exists an isomorphism of rational Hodge structures
$$
\psi\colon H^{2\sdot}(X_{I_1},\bbQ) \stackrel{\sim}{\lrarr} H^{2\sdot}(X_{I_2},\bbQ)
$$
that extends $\varphi$, respects the grading and the algebra structure.
\end{enumerate}
\end{thm}
\begin{proof}
Assume that $\varphi\in \JJ$, and let $\psi\in\GG(X)$ be a preimage of $\varphi$, see Definition \ref{defn_subgr}.
The action of $\psi$ respects the algebra structure and the grading by the definition of $\GG(X)$.
It remains to check that $\psi$ is a morphism of Hodge structures.
The complex structures $I_1$ and $I_2$ can be completed to a pair of hyperk\"ahler
structures $I_1$, $J_1$, $K_1$ and $I_2$, $J_2$, $K_2$, see section \ref{sec_lie_action}. Consider the operators $W_{I_1}$ and $W_{I_2}$
from Proposition \ref{prop_alg1}. The action of these operators on differential forms descends
to cohomology, so let us denote by $w_1$ and $w_2$ the corresponding endomorphisms of $H^\sdot(X,\bbC)$.
The endomorphisms $w_1$ and $w_2$ are the Weil operators that induce the Hodge decomposition on the cohomology.
It follows from Proposition \ref{prop_alg1}, that $w_1, w_2\in \frso(V_\bbR,q)\subset \frgt(X)^0_\bbR$.

By our assumptions, $\varphi$ is a morphism of Hodge structures. In terms of the Weil operators,
this means $\mathrm{ad}_\varphi(w_1) = w_2$. By Proposition \ref{prop_grhomo} the adjoint action of $\psi$ on $\frso(V,q)$ is
determined by the action of its degree two component, which equals $\varphi$. Hence we have
$\mathrm{ad}_\psi(w_1) = w_2$. This shows that the components $\psi_k$ in every degree $k$
are morphisms of Hodge structures. This proves the first part of the theorem. The proof of the
second part is analogous.
\end{proof}

\begin{cor}\label{cor_main}
Let $I_1$ and $I_2$ be two complex structures of hyperk\"ahler type on a compact
simply-connected manifold $X$. Assume that $I_1$ and $I_2$ define the same Hodge
structure on $H^2(X,\bbQ)$. Then they define the same Hodge structure on $H^k(X,\bbQ)$
for all $k$.
\end{cor}
\begin{proof}
Apply the previous theorem to $\varphi = \mathrm{Id}$, and note that
in its proof we can choose $\psi = \mathrm{Id}$.
\end{proof}

%\section*{Acknowledgements}


\begin{thebibliography}{HKL}

\bibitem[Be]{Be}
A.\ Beauville, {\it Vari\'et\'es K\"ahl\'eriennes dont la premi\`ere classe de Chern est nulle}, J. Diff. Geom. 18, 755--782 (1983)

\bibitem[B1]{B1}
F.\ Bogomolov, {\it On the cohomology ring of a simple hyperk\"ahler manifold (on the results of Verbitsky)},
Geom. Funct. Anal 6(4), 612--618 (1996)

\bibitem[B2]{B2}
F.\ Bogomolov, {\it The decomposition of K\"ahler manifolds with a trivial canonical class},
Mat. Sb. (N.S.) 93(135) (1974), 573--575, 630

\bibitem[GHJ]{GHJ}
M.\ Gross, D.\ Huybrechts, D.\ Joyce, {\it Calabi-Yau manifolds and related geometries.}
In: Lectures at a Summer School in Nordfjordeid, Norway, June 2001, Universitext. Berlin, Springer (2003)

\bibitem[HS]{HS}
N.\ Hitchin, J.\ Sawon, {\it Curvature and characteristic numbers of hyper-K\"ahler manifolds},
Duke Math. J. 106 (2001), 599--615

\bibitem[H1]{H1}
D.\ Huybrechts, {\it Compact hyperk\"ahler manifolds: basic results}, Invent. Math. 135
(1999), 63--113. Erratum in: Invent. Math. 152 (2003), 209--212

\bibitem[H2]{H2}
D.\ Huybrechts, {\it Finiteness results for compact hyperk\"ahler manifolds},
J. Reine Angew. Math. 558 (2003), 15--22

\bibitem[H3]{H3}
D.\ Huybrechts, {\it A global Torelli theorem for hyperk\"ahler manifolds [after M. Verbitsky]},
Ast\'erisque (2012), no. 348, Exp. No. 1040, x, 375--403, S\'eminaire Bourbaki: Vol. 2010/2011. Expos\'es 1027--1042.

\bibitem[Kn]{Kn}
M.\ Kneser, {\it Quadratische {F}ormen},
Revised and edited in collaboration with Rudolf Scharlau, Springer-Verlag, Berlin, 2002, viii+164 pp.

\bibitem[KSV]{KSV}
N.\ Kurnosov, A.\ Soldatenkov, M.\ Verbitsky, {\it Kuga-Satake construction and cohomology of hyperk\"ahler manifolds},
Adv. Math. 351 (2019), 275--295

\bibitem[LL]{LL}
E.\ Looijenga, V.\ Lunts, {\it A Lie algebra attached to a projective variety}, Invent. Math. 129(2), 361--412 (1997)

\bibitem[Ma]{Ma}
E.\ Markman, {\it A survey of Torelli and monodromy results for holomorphic-symplectic varieties},
Complex and differential geometry, 257--322, Springer Proc. Math., 8, Springer, Heidelberg, 2011.

\bibitem[Na]{Na}
Y.\ Nagai, {\it Birational geometry of O'Grady's six dimensional example over the Donaldson--Uhlenbeck compactification},
Math. Ann. 358 (2014), 143--168

\bibitem[Ni]{Ni}
M.\ Nieper, {\it Hirzebruch-Riemann-Roch formulae on irreducible symplectic K\"ahler manifolds},
J. Algebraic Geom. 12 (2003), no. 4, 715--739

\bibitem[No]{No}
S.\ Novikov, {\it Topological invariance of rational classes of Pontrjagin},
Dokl. Akad. Nauk SSSR, 163 (1965), 298--300

\bibitem[V1]{V2}
M.\ Verbitsky, {\it Action of the Lie algebra of $SO(5)$ on the cohomology of a hyper-K\"ahler manifold},
Funct. Anal. Appl. 24 (1990), no. 3, 229--230

\bibitem[V2]{Ve}
M.\ Verbitsky, {\it Cohomology of compact hyperk\"ahler manifolds},
Ph.D. dissertation, Harvard University, Cambridge, Mass., 1995; arXiv:alg-geom/9501001

\bibitem[V3]{V3}
M.\ Verbitsky, {\it Mirror symmetry for hyper-K\"ahler manifolds},
Mirror symmetry, III (Montreal, PQ, 1995), 115--156,
AMS/IP Stud. Adv. Math., 10, Amer. Math. Soc., Providence, RI, 1999.

\bibitem[V4]{V4}
M.\ Verbitsky, {\it Mapping class group and a global Torelli theorem for hyperk\"ahler manifolds},
Duke Math. J. 162 (2013), no. 15, 2929--2986.

\end{thebibliography}
\end{document}